\documentclass[11pt]{article}
\usepackage{amsmath,amsthm,verbatim,amssymb,bbm}
\usepackage{graphicx}
\usepackage{caption}

\graphicspath{ {plots/} }
\usepackage{float}
\usepackage{enumerate}
\usepackage{color}

\newtheorem{thm}{Theorem}[section]
\newtheorem{theorem}[thm]{Theorem}
\newtheorem{prop}[thm]{Proposition}
\newtheorem{lemma}[thm]{Lemma}

\newtheorem{cor}[thm]{Corollary}    

\newtheorem{example}[thm]{Example}

\def\mathd{\mathrm{d}}
\title{Ratios of Entire functions and generalized Stieltjes functions}
\author{Dimitris Askitis and Henrik L.\ Pedersen}

\begin{document}

\maketitle

\begin{abstract}
%
%
%
%
%
%

Monotonicity properties of the ratio 
$$
\log \frac{f(x+a_1)\cdots f(x+a_n)}{f(x+b_1)\cdots f(x+b_n)},
$$
where $f$ is an entire function are investigated. Earlier results for Euler's gamma function and other entire functions of genus 1 are generalised to entire functions of genus $p$ with negative zeros. Derivatives of order comparable to $p$ of the expression above are related to generalised Stieltjes functions of order $p+1$. Our results are applied to the Barnes multiple gamma functions. We also show how recent results on the behaviour of Euler's gamma function on vertical lines can be sharpened and generalised to functions of higher genus. Finally a connection to the so-called Prouhet-Tarry-Escott problem is described.
\end{abstract}
\noindent {\em \small 2020 Mathematics Subject Classification: Primary: 26A48,  Secondary: 30E20, 44A10} 

\noindent {\em \small Keywords: Entire function, Laplace transform, Generalized Stieltjes function, Euler's Gamma function, Barnes G-function}

\section{Introduction and results}

We revisit monotonicity properties of ratios of gamma functions and more generally entire functions of finite genus with translated variables. In 1986, Ismail and Bustoz showed that the function
\begin{equation}
\label{eq:bi}
x\mapsto \frac{\Gamma(x)\Gamma(x+a+b)}{\Gamma(x+a)\Gamma(x+b)}, \quad x>0
\end{equation}
is logarithmically completely monotonic for any $a,b>0$ (see \cite{bustoz-ismail}). Later results involve ratios of the form
$$
x\mapsto \frac{\Gamma(x+a_1)\cdots \Gamma(x+a_n)}{\Gamma(x+b_1)\cdots\Gamma(x+b_n)}, \quad x>0
$$
as well as situations where the number of factors in the numerator and denominator differ, and also with weighted variables (e.g.\ $x+a_j$ replaced by $A_jx+a_j$ and similarly for $x+b_j$). See e.g.~\cite{ismail1}, \cite{grinshpan-ismail1} and \cite{kp1}.

Recently the logarithm of the function in \eqref{eq:bi} was shown to be a so-called generalised Stieltjes function of order $2$, and hence also logarithmically completely 
monotonic (see \cite{bkp}). The purpose of this paper is to obtain similar results for sequences $(a_1,\ldots,a_n)$ and $(b_1,\ldots,b_n)$ in the framework of entire functions of 
finite genus having only real and non-positive zeros. 

A completely monotonic function $\varphi$ is a $C^{\infty}$-function defined on $(0,\infty)$ such that $(-1)^n\varphi^{(n)}(x)>0$ for all $n=0,1,\ldots$ and all $x>0$. These functions are characterised via Bernstein's theorem as the Laplace transforms of positive measures. A logarithmically completely monotonic function is a positive function $\varphi$ for which $-(\log \varphi)'$ is completely monotonic. These functions are also completely monotonic and can be characterised as those completely monotonic functions for which any $n$th root is also completely monotonic. For background on these functions see e.g.\ \cite{schilling} and \cite{steutel-van-harn}. A function $f$ defined on $(0,\infty)$ is called a generalised Stieltjes function of order $\lambda$ if it admits the representation 
$$
f(x)=\int_0^{\infty}\frac{\mathd \mu(t)}{(t+x)^{\lambda}}+c, \quad x>0,
$$
for some positive measure $\mu$ making the integral converge and for some $c\geq 0$. It is known, and easy to verify, that a function is a generalised Stieltjes function of order $\lambda$ if and only if it is the Laplace transform of $t^{\lambda-1}\varphi(t)$, where $\varphi$ is completely monotonic.

Denote by
$\mathcal E_p$ the class of real entire functions of genus $p$ with only real and non-positive zeros, and with $f(x)>0$ for $x>0$. For a function $f\in \mathcal E_p$ we arrange its zeros $\{-\lambda_k\}$ such that 
$0\leq\lambda_1\leq \lambda_2\leq \cdots$, and notice $\sum_{k=r+1}^{\infty}\lambda_k^{-p-1}<\infty$, where $r$ is the order of its possible zero at the origin. The function admits the Hadamard representation
$$
f(z)=e^{q(z)}P(z),
$$
where $q$ is a (real) polynomial of degree at most $p$, and
$$
P(z)=z^r\prod_{k=r+1}^{\infty}\left(1+z/\lambda_k\right) e^{-z/\lambda_k+\cdots+(-1)^pz^p/\lambda_k^p}.
$$
The Laplace representation of the derivative $(-1)^p(\log f)^{(p+1)}$  can be expressed as
\begin{equation}\label{eq:hlp-lemma}
(-1)^p(\log f)^{(p+1)}(z)=\int_{0}^\infty e^{-sz}s^ph(s)\mathd s,\quad \Re z>0, 
\end{equation}
where 
\begin{equation}
h(s)=\sum_{k=1}^\infty e^{-\lambda_ks}.\label{eq:def-h} 
\end{equation}
(See e.g.\ \cite[Proposition 2.1]{hlp}.)

We shall study ratios of products of the form
\begin{align}
\label{eq:defW}
W_f(x)=\frac{\prod_{k=1}^{n}f(x+a_k)}{\prod_{k=1}^{n}f(x+b_k)},
\end{align}
for $f\in \mathcal E_p$. The numbers $a_1,\ldots,a_n$ and $b_1,\ldots,b_n$ are non-negative and (without loss of generality) we assume $a_1\leq\cdots \leq a_n$ and $b_1\leq \cdots\leq b_n$. 

We shall define two functions in terms of these two sequences, namely
\begin{align}
g(s)&=\frac{1}{s^2}\sum_{k=1}^n \left(e^{-sa_k}-e^{-sb_k}\right),\label{eq:def-g}\\
\rho(t) &= \sum_{k=1}^n (t-a_k)\mathbbm{1}_{[a_k,\infty)}(t)-\sum_{k=1}^n (t-b_k)\mathbbm{1}_{[b_k,\infty)}(t). \label{eq:def-rho}
\end{align}
Here, $\mathbbm{1}_{S}$ denotes the indicator function of $S$. It turns out that monotonicity properties of $W_f$ are related with properties of the functions $g$ and $\rho$. Furthermore $g$ is completely monotonic if and only if $\rho$ is non-negative, and this holds if and only if 
\begin{align}
\label{eq:weak}
\sum_{k=1}^ma_k\leq\sum_{k=1}^{m}b_k, \quad \text{for all}\ m\in\{1,\ldots,n\}.
\end{align}
See Lemma \ref{lemma_wsmcm}.
We say that $b$ is a \textit{weak supermajorisation} of $a$ and denote it by $b\prec_wa$ if \eqref{eq:weak} holds.

Our main results deal with generalised Stieltjes functions and weak supermajorisation, and they are proved in Section \ref{sec:proofs}. 
There we also describe how one may in some situations obtain a generalised Stieltjes function without assuming weak supermajorisation. 
We consider the behaviour on vertical lines in Section \ref{sec:vertical} and show in particular how monotonicity results for Euler's gamma function on vertical lines in the complex plane can be obtained as corollaries of Theorems \ref{thm:main1} and \ref{thm:main2}. In Section \ref{sec:escott} we describe a connection with the so-called Prouhet-Tarry-Escott Problem.

To ease notation we denote 
by $\partial_z$ the usual derivative with respect to $z$.
\begin{theorem}
\label{thm:main1}
The function $(-1)^{p-1}\partial_z^p\log W_f(z)$ has the representation
\begin{align}
	(-1)^{p-1}\partial_z^p\log W_f(z)&=\int_0^\infty e^{-sz}s^{p+1}h(s)g(s)\mathd s\\
	&= \Gamma(p+2)\int_0^\infty \frac{\phi(t)\mathd t}{(z+t)^{p+2}},
\end{align}
	where $h$ is defined in \eqref{eq:def-h}, $g$ in \eqref{eq:def-g} and
	\begin{equation}
	 \label{eq:def-phi}
	 \phi(t)=\sum_{m=1}^\infty \rho(t-\lambda_m),
	\end{equation}
$\rho$ being defined in \eqref{eq:def-rho}.

Moreover, if $b\prec_wa$, then $(-1)^{p-1}\partial_z^p\log W_f(z)$ is a generalised Stieltjes function of order $p+2$.
\end{theorem}

\begin{theorem}
\label{thm:main2} Suppose that $p\geq 1$ and $\sum_{k=1}^na_k=\sum_{k=1}^nb_k$.  Then,
\begin{align}
	(-1)^{p}\partial_z^{p-1}\log W_f(z)&=\int_0^\infty e^{-sz}s^{p}h(s)g(s)\mathd s\label{eq:main2-1}\\
	&= \Gamma(p+1)\int_0^\infty \frac{\phi(t)\mathd t}{(z+t)^{p+1}},\label{eq:main2-2}
\end{align}	where $h$, $g$ and $\phi$ are as above.

Moreover, if $b\prec_wa$, the function $(-1)^{p}\partial_z^{p-1}\log W_f(z)$ is a generalised Stieltjes function of order $p+1$.
\end{theorem}
When $p=1$ and $f$ is the reciprocal of Euler's gamma function  we obtain a generalization of \cite[Proposition 4.1]{bkp}.
We remark that the representation in \eqref{eq:main2-2} implies that $\partial_z^{p-1}\log W_f(z)$ tends to $0$ as $z$ tends to infinity along any ray different from the negative real line.
In the case where $f$ is the reciprocal of Euler's gamma function, $\log W_f(x)$ can, using Stirling's series, be shown to tend to zero as $x$ tends to infinity when 
 the non-negative sequences $a$ and $b$ satisfy
$$
\sum_{j=1}^na_j=\sum_{j=1}^nb_j.
$$
For a general function $f$ in $\mathcal E_1$, $\log W_f(x)$ can be shown to tend to zero by using the representation of the so-called Pick-functions, see e.g.\ \cite{bkp}. 
In case of higher genus there are asymptotic formulae for entire functions with non-positive zeros having a density (see \cite{levin}), but Stirling type series for general functions in $\mathcal E_p$ are not known to us. In Lemma 2.4 we show how cancellation between the sequences $a$ and $b$ allows us to obtain the limit behaviour of $\log W_f$ and its derivatives for general $f\in \mathcal E_p$.

If we assume that $b\prec_wa$, and both of the relations 
$$\sum_{k=1}^na_k=\sum_{k=1}^nb_k,\quad \sum_{k=1}^na_k^2=\sum_{k=1}^nb_k^2
$$
hold then it turns out that $a=b$, see Proposition \ref{lemma:impossible}. Hence, Theorem \ref{thm:main2} does not give any infomation with these three assumptions. However, 
a primitive (up to a sign) 
of the function in Theorem \ref{thm:main2} may also be a generalised Stieltjes function of order $p+1$ under the assumption of non-negativity of a primitive of $\rho$. This is the contents of the next result, extending Theorem \ref{thm:main2}.


\begin{theorem}\label{thm:main3}Suppose that $p\geq 1$ and $1\leq \ell\leq p$. Let $a$ and $b$ satisfy
	$$
	\sum_{k=1}^na_k^j=\sum_{k=1}^nb_k^j
	$$
	for all $j$ with  $0\leq j\leq \ell$. Then,
	\begin{align}
	(-1)^{p-\ell+1}\partial_z^{p-\ell}\log W_f(z)&=\int_0^\infty e^{-sz}s^{p}h(s)g_\ell(s)\mathd s\\
	&= \Gamma(p+2)\int_0^\infty \frac{\phi_\ell(t)\mathd t}{(z+t)^{p+1}},
	\end{align}
	where
	\begin{align}
	g_\ell(s)&=\dfrac{\sum_{k=1}^n e^{-sa_k}-\sum_{k=1}^ne^{-sb_k}}{s^{1+\ell}}\quad \text{and}\\
	\phi_\ell(t)&=\sum_{m=1}^\infty \rho_\ell(t-\lambda_m)
	\end{align}
	and where
	\begin{align}
	\rho_\ell(t) = \sum_{k=1}^n (t-a_k)^{\ell}\mathbbm{1}_{[a_k,\infty)}(t)-\sum_{k=1}^n (t-b_k)^\ell\mathbbm{1}_{[b_k,\infty)}(t)\,.
	\end{align}
	Thus, if $\rho_\ell$ is non-negative,
	$(-1)^{p-\ell+1}\partial_z^{p-\ell}\log W_f(z)$
	is a generalised Stieltjes function of order $p+1$.
\end{theorem}
Proposition \ref{prop:sufficient} gives a sufficient condition for non-negativity of the function $\rho_{\ell}$ in Theorem \ref{thm:main3}
depending on the number of points in the sequences $a$ and $b$.

Letting  $a=(0,4,5)$ and $b=(1,2,6)$, the corresponding function $\rho_2$ is in fact non-negative (see Figure \ref{fig:rho_2b}), and so Theorem \ref{thm:main3} can be applied. We notice that in this case  $\rho_1=\rho$ takes both positive and negative values, so Theorem \ref{thm:main2} cannot be applied.
\begin{table}[H]
	\begin{center}
		\begin{tabular}{cc}
			\includegraphics[width=0.45\textwidth, keepaspectratio]{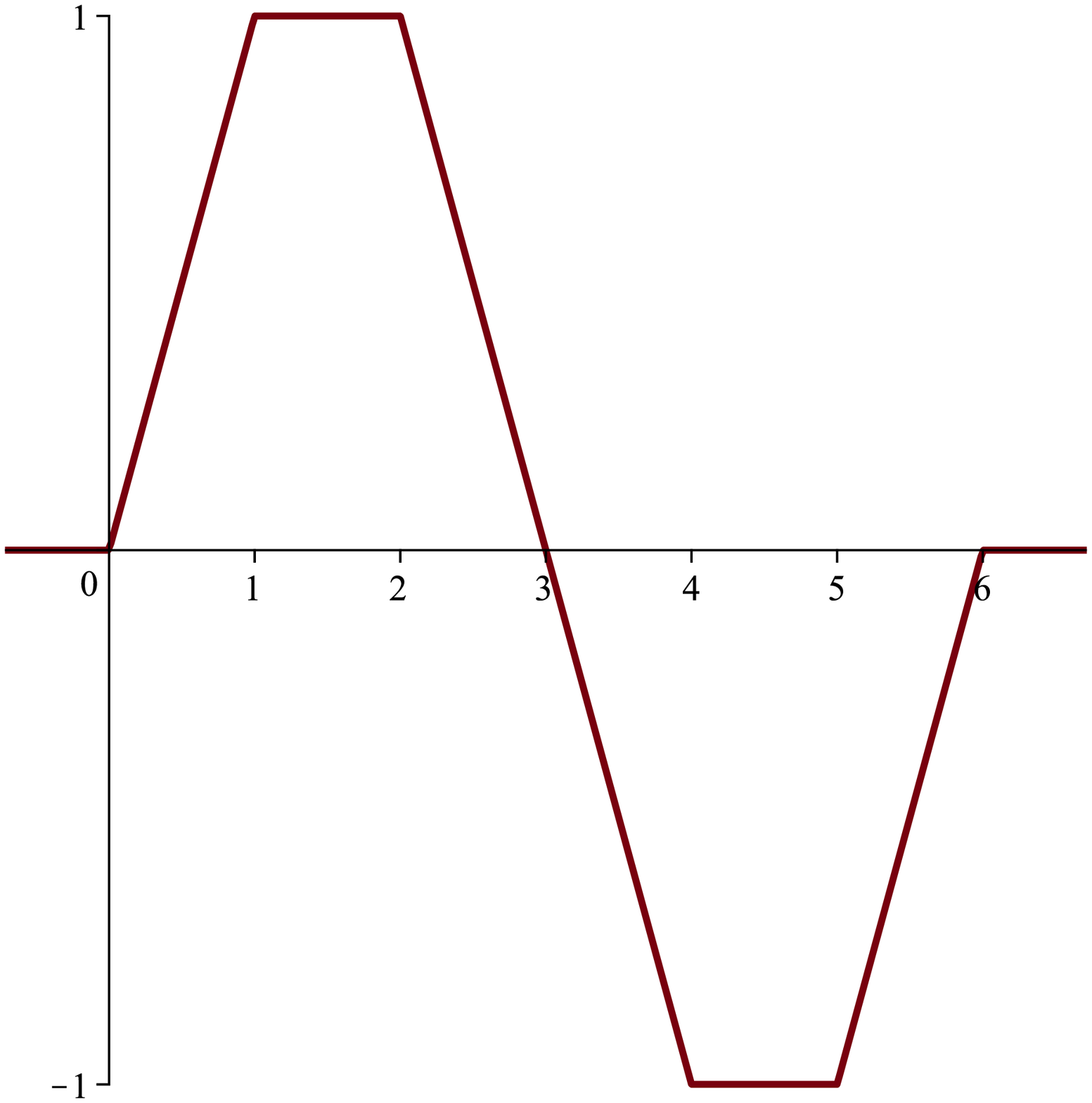}&
			\includegraphics[width=0.45\textwidth, keepaspectratio]{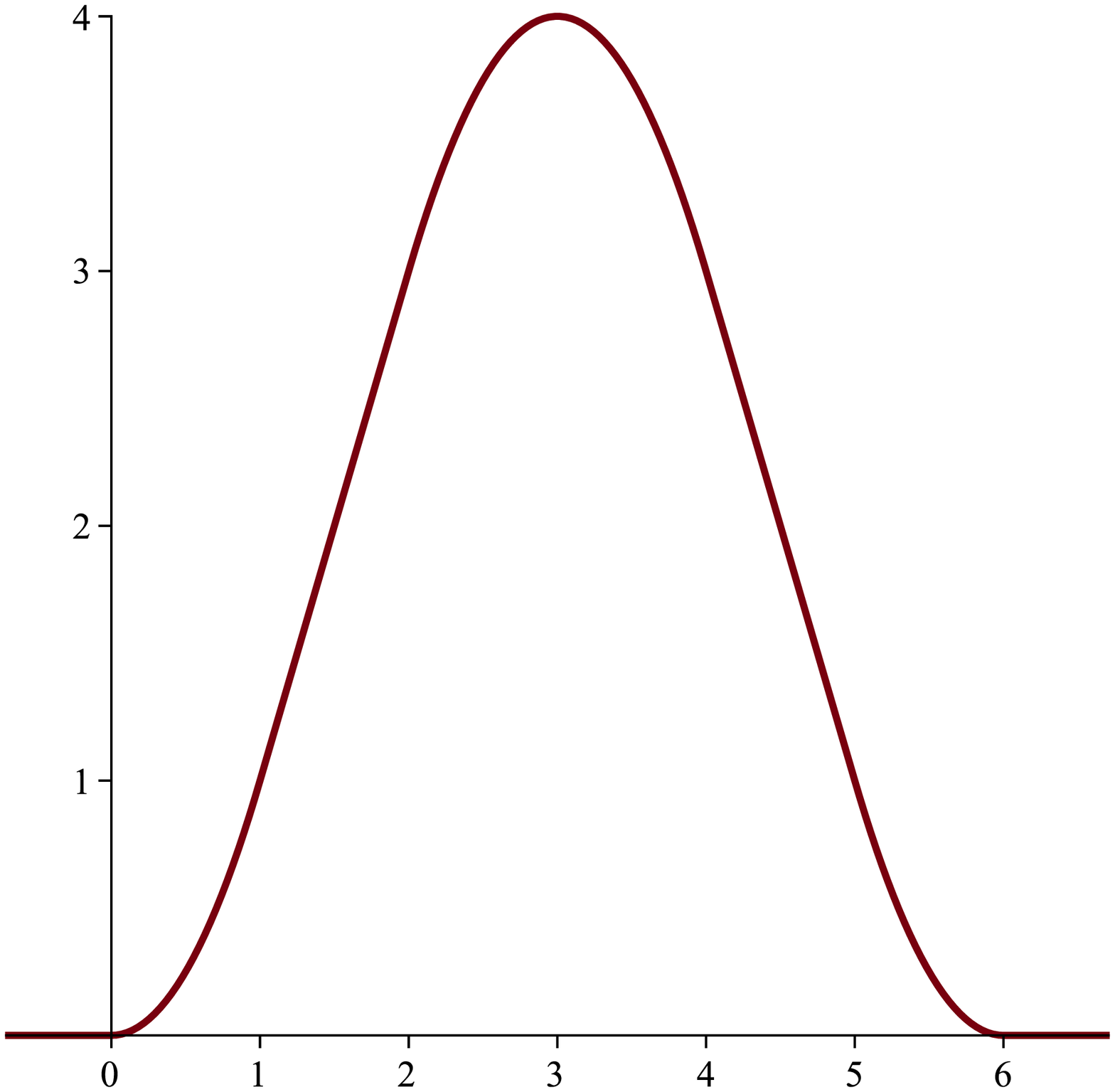}\\
			\scriptsize{$\rho_1$}&\scriptsize{$\rho_2$}
		\end{tabular}
		\captionof{figure}{$\rho_1$ and $\rho_2$, corresponding to $a=(0,4,5)$ and $b=(1,2,6)$.}
		\label{fig:rho_2b}
	\end{center}
\end{table}
Our results can be applied to the Barnes multiple gamma functions $\Gamma_N$ introduced in a series
of papers, see e.g.~\cite{Barnes}.  For the reader's convenience we briefly describe the construction of these functions (with parameters equal to $1$) following \cite{RUIJSENAARS}. Hurwitz' multiple zeta function $\zeta_N$ is given as
$$
\zeta_N(s,z)=\sum_{m_1,\ldots,m_N\geq 0}(z+m_1+\ldots+m_N)^{-s}.
$$ 
This function has as a function of $s$ a meromorphic extension to $\mathbb C$ and
the multiple gamma function is then defined as 
$$
\Gamma_N(z)=\exp \partial_s\zeta_N(s,z)_{s=0}.
$$
These functions satisfy the relations $\Gamma_{N+1}(z)=\Gamma_N(z)\Gamma_{N+1}(z+1)$ and $\Gamma_0(z)=1/z$.
The reciprocal $1/\Gamma_N(z)$ is an entire
function of genus $N$ having non-positive zeros located at $z=-k$, $k\in \{0,1,\ldots\}$, of multiplicity $\binom{k+N-1}{N-1}$, and $\Gamma_N(1)>0$. 

%
%

The reciprocal of Barnes' double gamma function $\Gamma_2$ is proportional to the so-called $G$-function defined by
$$
G(z+1)=(2\pi)^{z/2}e^{-((1+\gamma)z^2+z)/2}\prod_{k=1}^{\infty}\left(
  1+\frac{z}{k}\right) ^ke^{-z+z^2/2k}.
$$
(Here, $\gamma$ denotes Eulers constant.) It satisfies $G(1)=1$ and $G(z+1)=\Gamma(z)G(z)$. The zeros of $G(z+1)$ are at the negative integers and the multiplicity of the zero at $-k$ is $k$.

\begin{cor} 
\begin{align*}
	(-1)^{N-1}\partial_z^{N-1}\log W_{1/\Gamma_N}(z)&=\int_0^\infty e^{-sz}s^{N}h(s)g(s)\mathd s\\
	&= \Gamma(N+1)\int_0^\infty \frac{\phi(t)\mathd t}{(z+t)^{N+1}},
\end{align*}
where
$$
h(s)=\sum_{k=0}^{\infty}\begin{pmatrix}
                         k+N-1\\N-1
                        \end{pmatrix}
e^{-ks},
$$
and $g$ and $\phi$ are as in \eqref{eq:def-g} and \eqref{eq:def-phi}.
\end{cor}
The proof of this corollary follows from Theorem \ref{thm:main2}. 

\section{Proofs}
\label{sec:proofs}
In this section Theorems \ref{thm:main1} and \ref{thm:main2} are proved. The following lemma is an immediate consequence of the relation \eqref{eq:hlp-lemma}.
\begin{lemma}
\label{lemma:int-hlp}
	Let $f\in \mathcal E_p$ and assume $w_1,w_2\in\{\Re w>0\}$. Then,
	\[(-1)^p\partial_z^p\log f(w_2)-(-1)^p\partial_z^p\log f(w_1)=\int_0^\infty s^{p-1}h(s)(e^{-sw_1}-e^{-sw_2})\mathd s\]
	where $h$ is defined in \eqref{eq:def-h}. 
\end{lemma}
The next lemma is the key to our results.
\begin{lemma}\label{lemma_wsmcm}
	The function $g$ in \eqref{eq:def-g} is the Laplace transform of $\rho$ in \eqref{eq:def-rho}:
		\begin{align*}
	g(s)=\int_0^\infty e^{-st}\rho(t)\mathd t.
	\end{align*}
	Moreover, $g$ is completely monotonic if and only if $\rho$ is non-negative which is the case exactly when $(b_k)_{k=1}^n\prec_w(a_k)_{k=1}^n$.

\end{lemma} 
We remark that in the case of $n=2$ and $a_1=0$, $b_1=a$, $b_2=b$ and $a_2=a+b$ complete monotonicity of the corresponding function $g$ is immediate since 
$$
g(s)=\frac{1}{s^2}\left(1-e^{-(a+b)s}+e^{-as}-e^{-bs}\right)=\frac{1-e^{-as}}{s}\frac{1-e^{-bs}}{s}.
$$

\begin{proof}[Proof of Lemma \ref{lemma_wsmcm}]
Let $\mathd \sigma=\sum_{k=1}^n\left(\mathd\epsilon_{a_k}-\mathd\epsilon_{b_k}\right)$, where $\epsilon_a$ is the point mass at $a$, and $\mathd \tau(t)=v(t)\mathd t$, with $v(t)=t\mathbbm{1}_{[0,\infty)}(t)$.
Then 
$$
g=L(\sigma)L(\tau)=L(\sigma\ast \tau),
$$
and 
\begin{align*}
\mathd( \sigma\ast \tau)(t)&=
\left(\sum_{k=1}^n v(t-a_k)-\sum_{k=1}^nv(t-b_k)\right)\mathd t\\
&=
\left(\sum_{k=1}^n (t-a_k)\mathbbm{1}_{[a_k,\infty)}(t)-\sum_{k=1}^n (t-b_k)\mathbbm{1}_{[b_k,\infty)}(t)\right) \mathd t
\\&= \rho(t)\mathd t\,.
\end{align*}
This yields the relation
\begin{align*}
g(s)=\int_0^\infty e^{-st}\rho(t)\mathd t.
\end{align*}
According to Bernstein's theorem the function $g$ is completely monotonic if and only if $\rho$ is non-negative, which is the case if and only if $(b_k)_{k=1}^n\prec_w(a_k)_{k=1}^n$. To see this we may, for given $t$, choose $k_0$ and $l_0$ such that 
$a_{k_0}\leq t<a_{k_0+1}$, and $b_{l_0}\leq t<b_{l_0+1}$. Then we can write
\begin{align*}
\rho(t)=
(k_0-l_0)t-\sum_{k=1}^{k_0} a_k+\sum_{k=1}^{l_0}b_k
\end{align*} and considering the cases $k_0<l_0$ and $k_0>l_0$ separately yields the result.
%
%
\end{proof}
Lemma \ref{lemma_wsmcm} is extended to the following result, using partial integration.
\begin{lemma}\label{lemma_wsmcm2}
	Let $\ell\in\mathbb{N}$. The function
	\[g_\ell(s)=\dfrac{\sum_{k=1}^n e^{-sa_k}-\sum_{k=1}^ne^{-sb_k}}{s^{1+\ell}}\]
	has integral representation
	\begin{align*}
	g_\ell(s)=\frac{1}{\ell!}\int_0^\infty e^{-st}\rho_\ell(t)\mathd t,
	\end{align*}
	where
	\begin{align}
	\rho_\ell(t) = \sum_{k=1}^n (t-a_k)^{\ell}\mathbbm{1}_{[a_k,\infty)}(t)-\sum_{k=1}^n (t-b_k)^\ell\mathbbm{1}_{[b_k,\infty)}(t)\,.
	\end{align}
	Thus, $g_\ell$ is completely monotonic if and only if $\rho_\ell$ is non-negative.
\end{lemma}

\begin{lemma}
 \label{lemma:limit}
 Let $f\in\mathcal E_p$ and $1\leq \ell \leq p$. Suppose that $a_1,\ldots,a_n,b_1,\ldots,b_n$ are non-negative and that $\Delta_m=0$ for all $m$ with $0\leq m\leq \ell$, where
 $$
 \Delta_m=\sum_{j=1}^n\left(a_j^m-b_j^m\right), \quad m\geq 0.
 $$
 Then for $x\geq 2\max\{a_n,b_n\}$ we have 
 $$
 \partial_x^{p-\ell}\log W_f(x)=\sum_{k=1}^{\infty}\frac{1}{(x+\lambda_k)^{p+1}}\sum_{m=\ell +1}^{\infty}\Delta_m\frac{C_{m,\ell}}{(x+\lambda_k)^{m-\ell-1}},
 $$                                 
 where
 $$C_{m,\ell} = \begin{cases}
 \begin{pmatrix}
 \ell-p\\m
 \end{pmatrix},\qquad \ell<p \\\\
 \dfrac{(-1)^{m+1}}{m}, \qquad \ell=p
 \end{cases}
 $$
 and
 $$
 \sum_{m=\ell +1}^{\infty}
 \Delta_m\frac{C_{m,l}}{(x+\lambda_k)^{m-\ell-1}}
 $$
 is bounded with a bound independent of $x$ and $k$.
\end{lemma}
\begin{proof}
 First we consider the case $\ell<p$. From the Hadamard representation of $f$ (cf.\ \cite{hlp}) it follows that
 \begin{align*}
 &\partial_x^{p-\ell}\log W_f(x)=\sum_{j=1}^n\left(q^{(p-\ell)}(x+a_j)-q^{(p-\ell)}(x+b_j)\right)\\
 &+\sum_{k=1}^{\infty}\left\{(-1)^{p-\ell-1}(p-\ell-1)!\sum_{j=1}^n\left(\frac{1}{(x+\lambda_k+a_j)^{p-\ell}}-\frac{1}{(x+\lambda_k+b_j)^{p-\ell}}\right)\right.\\
 &\hspace{1cm}\left. +\sum_{i=p-\ell}^p\frac{(-1)^i}{\lambda_k^i}\frac{\Gamma(i)}{\Gamma(i-p+\ell+1)}\sum_{j=1}^n\left((x+a_j)^{i-p+\ell}-(x+b_j)^{i-p+\ell}\right)\right\}.
 \end{align*}
 Next we notice that (by the binomial theorem), when $0\leq s\leq \ell$, 
 $$
 \sum_{j=1}^n\left((x+a_j)^s-(x+b_j)^s\right)=0.
 $$
 Therefore the sums over $j$ in the first and third line of the right hand side in the formula above are equal to $0$. Hence we have obtained
  \begin{align*}
 &\partial_x^{p-\ell}\log W_f(x)=\\
 &(-1)^{p-\ell-1}(p-\ell-1)!\sum_{k=1}^{\infty}\sum_{j=1}^n\left(\frac{1}{(x+\lambda_k+a_j)^{p-\ell}}-\frac{1}{(x+\lambda_k+b_j)^{p-\ell}}\right).
 \end{align*}
Let, for brevity, $c_n=\max\{a_n,b_n\}$. Now, assuming that $x\geq 2c_n$ it follows that 
\begin{align*}
\frac{1}{(x+\lambda_k+a_j)^{p-\ell}}&=\frac{1}{(x+\lambda_k)^{p-\ell}}\frac{1}{(1+a_j/(x+\lambda_k))^{p-\ell}}\\
&=\frac{1}{(x+\lambda_k)^{p-\ell}}\sum_{m=0}^{\infty}\begin{pmatrix}\ell-p\\m\end{pmatrix}\left(\frac{a_j}{x+\lambda_k}\right)^m,
\end{align*}
by the binomial series. This yields
 \begin{align*}
 &\frac{(-1)^{p-\ell-1}}{(p-\ell-1)!}\partial_x^{p-\ell}\log W_f(x)\\
 &=\sum_{k=1}^{\infty}\sum_{j=1}^n\left(\frac{1}{(x+\lambda_k+a_j)^{p-\ell}}-\frac{1}{(x+\lambda_k+b_j)^{p-\ell}}\right)\\
 &=\sum_{k=1}^{\infty}\frac{1}{(x+\lambda_k)^{p-\ell}}\sum_{m=0}^{\infty}\begin{pmatrix}\ell-p\\m\end{pmatrix}\left(\frac{1}{x+\lambda_k}\right)^m \sum_{j=1}^n\left(a_j^m-b_j^m\right)\\
 &=\sum_{k=1}^{\infty}\frac{1}{(x+\lambda_k)^{p-\ell}}\sum_{m=\ell+1}^{\infty}\begin{pmatrix}\ell-p\\m\end{pmatrix}\left(\frac{1}{x+\lambda_k}\right)^m \Delta_m\\
 &=\sum_{k=1}^{\infty}\frac{1}{(x+\lambda_k)^{p+1}}\sum_{m=\ell+1}^{\infty}\begin{pmatrix}\ell-p\\m\end{pmatrix}\left(\frac{1}{x+\lambda_k}\right)^{m-\ell-1} \Delta_m.
 \end{align*}
 It remains to be verified that the inner sum is bounded for $x\geq 2c_n$ with a bound independent of $k$. We get, using $|\Delta_m|\leq 2nc_n^m$
 \begin{align*}
  \lefteqn{\sum_{m=\ell+1}^{\infty}\left|\begin{pmatrix}\ell-p\\m\end{pmatrix}\left(\frac{1}{x+\lambda_k}\right)^{m-\ell-1} \Delta_m\right|}\\
  &\leq 2nc_n^{\ell+1}\sum_{m=\ell+1}^{\infty}\left|\begin{pmatrix}\ell-p\\m\end{pmatrix}\left(\frac{c_n}{x+\lambda_k}\right)^{m-\ell-1}\right| \\
  &\leq 2nc_n^{\ell+1}\sum_{m=\ell+1}^{\infty}\left|\begin{pmatrix}\ell-p\\m\end{pmatrix}\right|2^{-m+\ell+1}.
 \end{align*}
This last series converges and is independent of $k$.

In the case $\ell=p$, using that all polynomial terms vanish, we get
\begin{align*}
\log& W_f(x)\\
&= \sum_{k=1}^{\infty}\sum_{j=1}^n\left(\log (x+\lambda_k+a_j)-\log (x+\lambda_k+b_j) \right)\\
&=\sum_{k=1}^{\infty}\sum_{j=1}^n\left(
\sum_{m=1}^{\infty} \frac{(-1)^{m+1}}{m}\left(\frac{a_j}{x+\lambda_k}\right)^m
-\sum_{m=1}^{\infty} \frac{(-1)^{m+1}}{m}\left(\frac{b_j}{x+\lambda_k}\right)^m
\right)\\
&=\sum_{k=1}^{\infty}\left(
\sum_{m=1}^{\infty} \frac{(-1)^{m+1}}{m}
\frac{\Delta_m}{(x+\lambda_k)^m}
\right).
\end{align*}
 \end{proof}

\begin{proof}[Proof of Theorem \ref{thm:main1}]
Notice that 
$$\log W_f(z)=\sum_{k=1}^n\left(\log f(z+a_k)-\log f(z+b_k)\right).
$$
Hence for $\Re z>0$, by Lemma \ref{lemma:int-hlp},
	\begin{align*}
		(-1)^{p-1}\partial_z^p\log W_f(z)&= \int_0^\infty e^{-sz}s^{p-1}h(s)\left(\sum_{k=1}^n e^{-sa_k}-\sum_{k=1}^ne^{-sb_k}\right)\mathd s\\
		&=\int_0^\infty e^{-sz}s^{p+1}h(s)g(s)
		\mathd s,
	\end{align*}
	with $g$ as in \eqref{eq:def-g}.
We have, with $\tau=\sum_{m=1}^{\infty}\epsilon_{\lambda_m}$, from Lemma \ref{lemma_wsmcm},
$$
h(s)g(s)=L(\tau)(s)L(\rho)(s)=L(\tau\ast\rho)(s)=L(\phi)(s).
$$
This gives us
	\begin{align*}
	(-1)^{p-1}\partial_z^p\log W_f(z)
	&=\int_0^\infty e^{-sz}s^{p+1}\int_0^\infty e^{-st}\phi(t)\mathd t \mathd s\\
	&=\int_0^\infty \int_0^\infty s^{p+1} e^{-s(z+t)} \mathd s\phi(t)\mathd t\\
	&=\Gamma(p+2)\int_0^\infty \frac{\phi(t)\mathd t}{(z+t)^{p+2}}\,.
	\end{align*}
	If $b\prec_wa$ then, by Lemma \ref{lemma_wsmcm}, $\rho$ and hence also $\phi$ is non-negative and thus $(-1)^{p-1}\partial_z^p\log W_f(z)$ is a generalised Stieltjes function of order $p+2$.
\end{proof}

\begin{proof}[Proof of Theorem \ref{thm:main2}]
 With the extra assumption that $\sum_{k=1}^na_k=\sum_{k=1}^nb_k$ the function $g$ is bounded for $s$ near $0$. In fact, using l'Hospital's rule,
 $$
 \lim_{s\to 0}g(s)=\frac{1}{2}\sum_{k=1}^n\left(a_k^2-b_k^2\right).
 $$
 Therefore the function of $s$ (for fixed $x>0$) $e^{-sx}s^{p}h(s)g(s)$ is integrable on $[0,\infty)$ implying that 
 $$
 \Phi(x)=\int_0^{\infty} e^{-sx}s^{p}h(s)g(s)\mathd s
 $$
 is defined. Furthermore, 
 $$
 \Phi'(x)=-\int_0^{\infty} e^{-sx}s^{p+1}h(s)g(s)\mathd s=(-1)^p\partial_x^p\log W_f(x),
 $$
 showing that  
 $$
 (-1)^{p}\partial_x^{p-1}\log W_f(x)=\int_0^{\infty} e^{-sx}s^{p}h(s)g(s)\mathd s+c
 $$
 for a constant $c$. Lemma \ref{lemma:limit}, Lebesgue's theorem on dominated convergence together with the fact that $\sum_{m=1}^\infty 1/\lambda_m^{p+1}<\infty$ yield 
 \begin{align*}
 c=\lim_{x\to +\infty}(-1)^{p}\partial_x^{p-1}\log W_f(x)=0.
 \end{align*}
Arguing as in the proof of Theorem \ref{thm:main1} gives us
\begin{align*}
(-1)^{p}\partial_z^{p-1}\log W_f(z)=\Gamma(p+1)\int_0^\infty \frac{\phi(t)\mathd t}{(z+t)^{p+1}}\,,
\end{align*}
and if  $b\prec_wa$, $ (-1)^{p}\partial_z^{p-1}\log W_f(z)$ is a generalised Stieltjes function of order $p+1$. 
 \end{proof}

\begin{proof}[Proof of Theorem \ref{thm:main3}]
	The proof will be done by induction on $\ell$. For $\ell=1$, the result holds by Theorem \ref{thm:main2}.
	If the statement holds for $\ell-1$, we shall show that it holds for $\ell$. The conditions $\sum_{k=1}^{n}a_k^j=\sum_{k=1}^{n}b_k^j$ for all $j\leq \ell$ implies that the $g_\ell(s)=g(s)/s^{\ell-1}$ is bounded for $s$ near $0$.
	Therefore the function of $s$ (for fixed $x>0$) $e^{-sx}s^{p-\ell+1}h(s)g(s)$ is integrable on $[0,\infty)$ so that 
	$$
	\Phi(x)=\int_0^{\infty} e^{-sx}s^{p}h(s)g_\ell(s)\mathd s
	$$
	is defined. Furthermore, 
	$$
	\Phi'(x)=-\int_0^{\infty} e^{-sx}s^{p}h(s)g_{\ell-1}(s)\mathd s=(-1)^{p-\ell+1}\partial_x^{p-\ell+1}\log W_f(x),
	$$
	by the induction hypothesis. This shows that  
	$$
	(-1)^{p-\ell+1}\partial_x^{p-\ell}\log W_f(x)=\int_0^{\infty} e^{-sx}s^{p}h(s)g_\ell(s)\mathd s+c
	$$
	for a constant $c$. Using Lemma \ref{lemma:limit}, we get that $c=0$ and this gives
	\begin{align*}
	(-1)^{p-\ell+1}\partial_z^{p-\ell}\log W_f(z)=\Gamma(p+1)\int_0^\infty \frac{\phi_\ell(t)\mathd t}{(z+t)^{p+1}}\,.
	\end{align*}
	If $\rho_\ell$ is non-negative, so is $\phi_\ell$, and thus $(-1)^{p-\ell+1}\partial_z^{p-\ell}\log W_f(z  )$ is a generalised Stieltjes function of order $p+1$. This completes the induction step.
\end{proof}

Let us end this section by describing a variation of Theorem \ref{thm:main1} and Theorem \ref{thm:main2}. For a given finite sequence $a=(a_1,\ldots,a_n)$ of non-negative numbers and $f\in \mathcal{E}_p$ with zeros $\mathcal{Z}_f$,
we consider the sequence of zeros of the function 
$$
G_a(z)=\prod_{j=1}^nf(z+a_j),
$$
denoted by $\mathcal{Z}_{G_a}$, which is equal to $\{\rho-a_j\, |\, 1\leq j\leq n,\rho\in \mathcal{Z}_f\}$ counting multiplicities. We then denote by $a-\mathcal{Z}_f$ the non-decreasing rearrangement $(\mu_{a,k})_k$ of $-\mathcal{Z}_{G_a}$. For two sequences $a$ and $b$ we say that $b-\mathcal{Z}_f\prec_w a-\mathcal{Z}_f$ if
$$
\sum_{k=1}^m\mu_{a,k}\leq \sum_{k=1}^m\mu_{b,k}
$$
for all $m\geq 1$. 


  The key assumption is thus weak supermajorisation of these two infinite sequences. The result is stated as follows.
\begin{prop}
\label{prop:variant}
	\label{thm:main4} Suppose that $p\geq 1$ and $0\leq \ell\leq p$. If $\sum_{k=1}^na_k^j=\sum_{k=1}^nb_k^j$ for $0\leq j\leq \ell$ and if $b-\mathcal{Z}_f\prec_w a-\mathcal{Z}_f$, then
	$(-1)^{p-\ell+1}\partial_z^{p-\ell}\log W_f(z)$
	is a generalised Stieltjes function of order $p-\ell+2$ with the representation
	\begin{align}
	(-1)^{p-\ell+1}\partial_z^{p-\ell}\log W_f(z)&=\int_0^\infty e^{-sz}s^{p-\ell +1}h(s)g(s)\mathd s\\
	&= \Gamma(p-\ell+2)\int_0^\infty \frac{\phi(t)\mathd t}{(z+t)^{p-\ell+2}},
	\end{align}
	where  $g$ and $\phi$ are as in \eqref{eq:def-g} and \eqref{eq:def-phi}.
\end{prop}
\begin{proof}
The proof is similar to the proof of Theorem \ref{thm:main3}, except we consider
$$
\Phi(z)=\int_0^{\infty} e^{-sz}s^{p-\ell+1}h(s)g(s)\mathd s
$$
and show that it equals $\Gamma(p-\ell+2)\int_0^\infty \frac{\phi(t)\mathd t}{(z+t)^{p-\ell+2}}$.
\end{proof}
\begin{example}
\label{rem:reordering} 
We illustrate Proposition \ref{prop:variant} using the Gamma function and the $G$-function of Barnes.
\begin{enumerate}
 \item[(a)] 
Let $f(z)=1/\Gamma(z+1)$, $a=(0,3,3)$, and $b=(1,1,4)$. In this case,
\begin{align*}
 a-\mathcal{Z}_f&=(1,2,3,4,4,4,5,5,5,6,6,6,\ldots)\\
 b-\mathcal{Z}_f&=(2,2,3,3,4,4,5,5,5,6,6,6,\ldots)
\end{align*}
and it follows that $b'=b-\mathcal{Z}_f\prec_w a-\mathcal{Z}_f=a'$. (In fact, we have  $\sum_{k=1}^mb'_k-\sum_{k=1}^ma'_k=1$ for $m\leq3$ and $\sum_{k=1}^mb'_k-\sum_{k=1}^ma'_k=0$ for $m\geq 4$.) The assumptions in Proposition \ref{prop:variant} thus hold for $\ell=2$.
 The graph of $\phi(t)=\sum_{m=1}^{\infty}\rho(t-m)$ is shown below (case (a)) and $-\log W_{f}\in \mathcal S_2$. Since 
	\[-\log W_{f}(x)= -\log \frac{(x+1)(x+4)}{(x+2)(x+3)}\]
	the assertion can also be obtained directly, noticing that $-\log W_{f}(x)=\theta(x+2)-\theta(x+4)$, where $\theta(x)=\log(1+1/x)$ is an ordinary Stieltjes function.
	\item[(b)] Let $f(z)=G(z+1)$, $a=(0,3/2,3/2)$, and $b=(1/2,1/2,2)$. In this case $\mathcal{Z}_f=-(1,2,2,3,3,3,\ldots)$, and the assumptions in Proposition \ref{prop:variant} are satisfied for $\ell=2$. The graph of the function $\phi(t)=\sum_{m=1}^{\infty}m\rho(t-m)$ is shown below (case (b)) and $-\log W_f\in \mathcal S_2$. Furthermore,  
	$$
	-\log W_f(x)=-\log \frac{G(x+1)G(x+5/2)^2}{G(x+3/2)^2G(x+3)}
	=\log \frac{\Gamma(x+1)\Gamma(x+2)}{\Gamma(x+3/2)^2}
	$$
	from which it can also be seen that it is a generalised Stieltjes function of order $2$. 
\end{enumerate}
\end{example}
\begin{table}[H]
	\begin{center}
		\begin{tabular}{cc}
			\includegraphics[height=50mm]{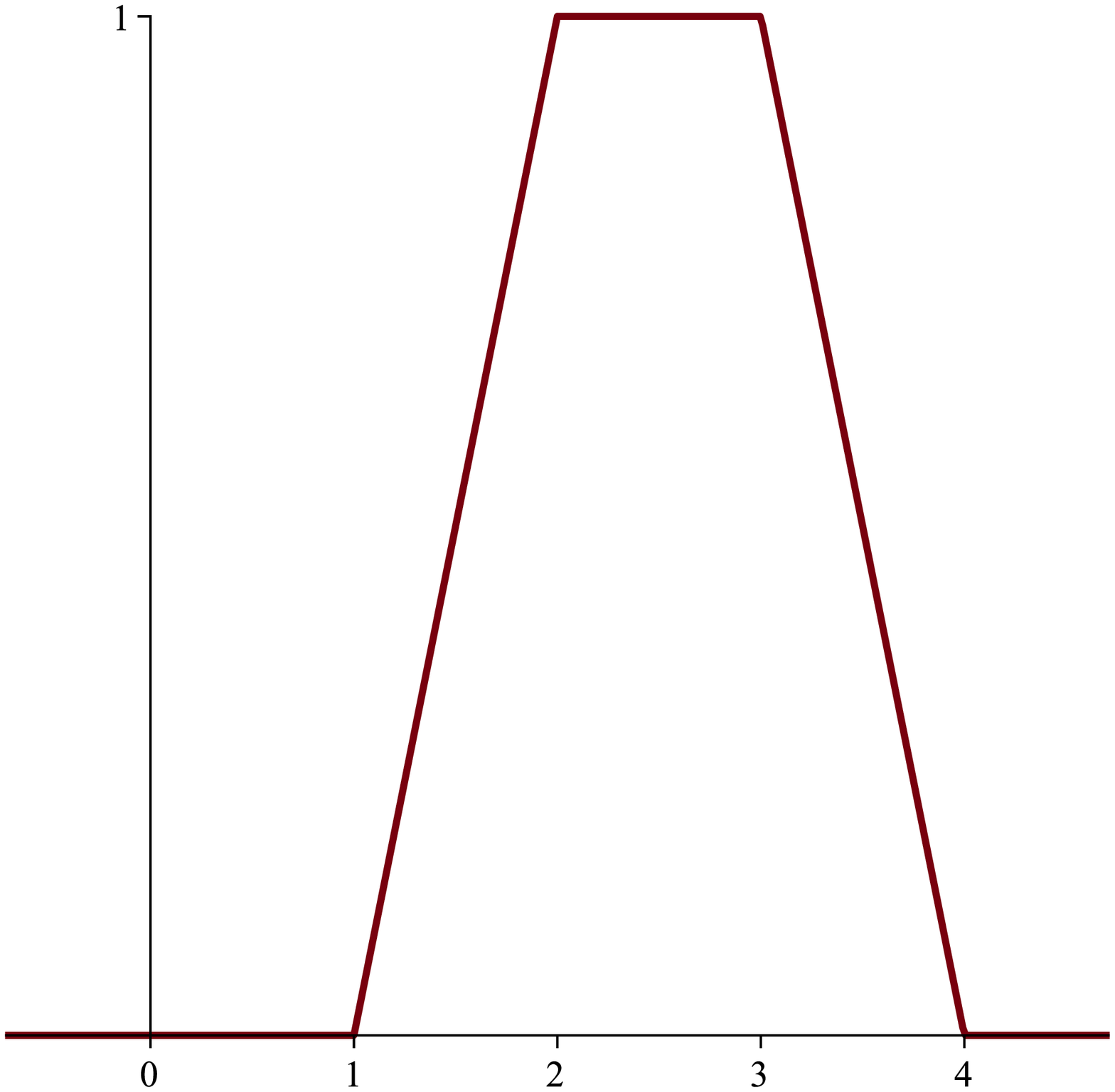}&
			\includegraphics[width=0.4\textwidth]{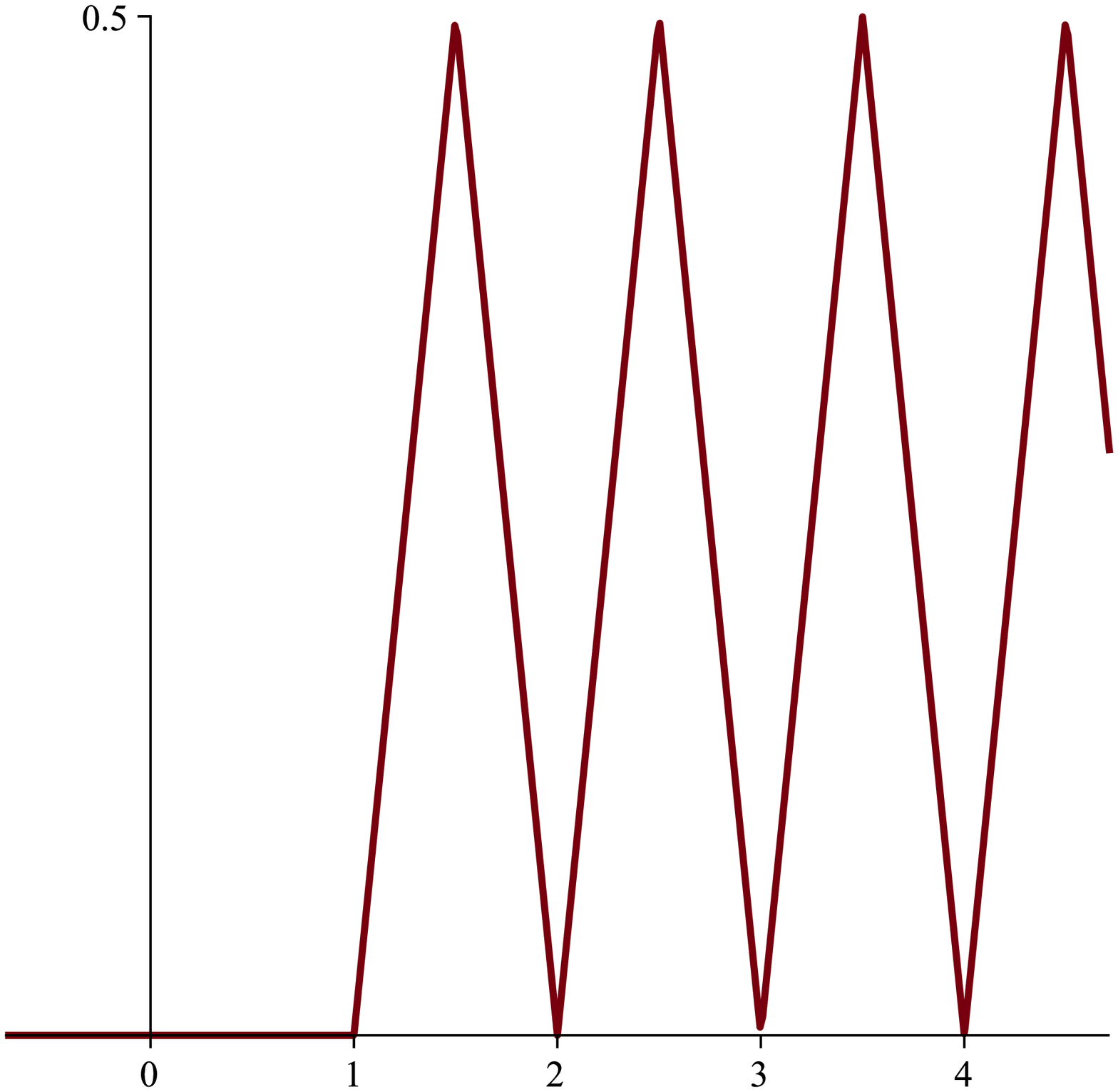}\\
			\scriptsize{$\phi$ in case (a)}&\scriptsize{$\phi$ in case (b)}
		\end{tabular}
		\captionof{figure}{The functions $\phi$ from Example \ref{rem:reordering} are non-negative.}
	\end{center}
\end{table}

\section{Behaviour on vertical lines}
\label{sec:vertical}
The motivation for the results in this section comes from the paper \cite{ismail1} in which  inequalities for the gamma (and $q$-gamma) function of complex arguments are investigated. One of the results in that paper states that the function $x\mapsto -\log | \Gamma(a+i\sqrt{x})|$ (where $a>0$)
has a completely monotonic derivative on $(0,\infty)$.

We investigate this type of questions for entire functions of finite genus with real zeros. Given such a function $f$, 
which monotonicity properties does the function 
$x\mapsto \log|f(a+i\sqrt{x})|$
have? 

Also in this section, $\epsilon_a$ denotes the point mass at $a$.
\begin{prop}
	\label{prop:vertical}
    Suppose that $f$ is a real entire function of genus $p$, and with only real 
	zeros $\{\lambda_k\}$.  Let 
	$$
    u(x)=\log|f(i\sqrt{x})|,
	$$
	and $m=\lfloor p/2\rfloor$ be the integer part of $p/2$. 
	Then $(-1)^m\partial_x^{m+1}u(x)$ is a generalized Stieltjes function of order $m+1$ with the representation
	$$
	(-1)^m\partial_x^{m+1}u(x)=\frac{m!}{2}\int_0^{\infty}\frac{\mathd \mu(t)}{(t+x)^{m+1}},
	$$
	where $\mu=\sum_{k=1}^{\infty}\epsilon_{\lambda_k^2}$.
\end{prop}
\begin{proof}
	From the Hadamard factorization of $f$ it easily follows that
	$$
	|f(i\sqrt{x})|^2=\kappa(x),\quad x\geq 0,
	$$
	where
	$$
	\kappa(z)=e^{r(z)}\prod_{k=1}^{\infty}\left(1+\frac{z}{\lambda_k^2}\right)\, e^{\sum_{j=1}^m(-z/\lambda_k^2)^j/j}
	$$
	is an entire function of the class $\mathcal E_m$, $r$ being a real polynomial of degree at most $m$. The result now follows directly from \cite[Proposition 2.1]{hlp}.
\end{proof}
The case concerning $f(a+i\sqrt{x})$ is immediate since  $z\mapsto f(z+a)$ is again entire of genus $p$ and with real zeros only: 
\begin{cor}
	\label{cor:vertical-a}
	Let $f$ be a real entire function of genus $p$, with only real 
	zeros $\{\lambda_k\}$ and let $m=\lfloor p/2\rfloor$. For $a\in \mathbb R$ we have
	%
	$$
	(-1)^m\partial_x^{m+1}\log|f(a+i\sqrt{x})|=\frac{m!}{2}\int_0^{\infty}\frac{\mathd\mu_a(t)}{(t+x)^{m+1}},
	$$
	where $\mu_a=\sum_{k=1}^{\infty}\epsilon_{(a-\lambda_k)^2}$.
\end{cor}
If $p\leq 1$ this corollary extends and sharpens the result mentioned in the beginning of this section, showing in particular that the derivative of $-\log|\Gamma(a+i\sqrt{x})|$ is a generalised Stieltjes function of order $2$, and hence completely monotonic. It is even logarithmically completely monotonic (see \cite{bkp}).  
\begin{cor}
	Let $f$ be a real entire function of genus $p$ and with only real 
	zeros $\{\lambda_k\}$, let $m=\lfloor p/2\rfloor$, $u(x)=\log|f(i\sqrt{x})|$ and $a>0$.
	Then the function $x\mapsto (-1)^m\partial_x^m(u(x+a)-u(x))$ is a generalized Stieltjes function of order $m+1$ with representation
	$$
	(-1)^m\partial_x^m(u(x+a)-u(x))=\frac{m!}{2}\int_0^{\infty}\frac{\mathd(\mu\ast m_a)(t)}{(x+t)^{m+1}},
	$$
	where $\mu=\sum_{k=1}^{\infty}\epsilon_{\lambda_k^2}$ and $m_a$ is Lebesgue measure on $(0,a)$.
\end{cor}
The proof of this corollary follows from the relation
	\begin{align*}
	\int_x^{x+a}\int_0^{\infty}\frac{\mathd\mu(s)}{(t+s)^{\lambda}}\, \mathd t=
	\int_0^{\infty}\int_0^{a}\frac{\mathd t}{(x+t+s)^{\lambda}}\mathd\mu(s)=\int_0^{\infty}\frac{\mathd(\mu\ast m_a)(t)}{(x+t)^{\lambda}}.
	\end{align*}
Let us finally mention a version of Theorem \ref{thm:main2} for vertical lines.
\begin{prop}
	\label{prop:vertical-w} Let $f$ be a real entire function of genus $p\geq 2$ and with only real 
	zeros $\{\lambda_k\}$, and let $m=\lfloor p/2\rfloor$. If $b\prec_wa$ and
	$$
	\sum_{k=1}^na_k=\sum_{k=1}^nb_k
	$$
	then
	\begin{align}
	(-1)^{m}\partial_x^{m-1} &\log \left|\frac{f(i\sqrt{x+a_1})\cdots f(i\sqrt{x+a_n})}{f(i\sqrt{x+b_1})\cdots f(i\sqrt{x+b_n})}\right| \label{eq:wg}
	\\
	&=\frac{1}{2}\int_0^\infty e^{-sx}s^{m}\xi(s)g(s)\mathd s,\nonumber
	\end{align}
	where $g$ is as in \eqref{eq:def-g} and $\xi(s)=\sum_{k=1}^{\infty}e^{-\lambda_k^2s}$. In particular the function in \eqref{eq:wg} is a generalised Stieltjes function of order $m+1$.
\end{prop}
\begin{proof}
	We have, where $\kappa$ is as in the proof of Proposition \ref{prop:vertical},
	\begin{align}
	\log \left|\frac{f(i\sqrt{x+a_1})\cdots f(i\sqrt{x+a_n})}{f(i\sqrt{x+b_1})\cdots f(i\sqrt{x+b_n})}\right| 
	=\frac{1}{2}\log W_{\kappa}(x),
	\end{align}
	and the result follows from Theorem \ref{thm:main2}.
\end{proof}

Returning to the $G$-function of Barnes we notice the following corollary. 
\begin{cor}If $b\prec_wa$ and
	$
	\sum_{k=1}^na_k=\sum_{k=1}^nb_k
	$
	then
 \begin{align*}
	-&\log \left|\frac{G(i\sqrt{x+a_1})\cdots G(i\sqrt{x+a_n})}{G(i\sqrt{x+b_1})\cdots G(i\sqrt{x+b_n})}\right| 
	=\frac{1}{2}\int_0^\infty e^{-sx}s\xi(s)g(s)\mathd s,\nonumber
	\end{align*}
	 is a generalised Stieltjes function of order $2$. Here $g$ is as in \eqref{eq:def-g} and $\xi(s)=\sum_{k=0}^{\infty}(k+1)e^{-k^2s}$. 
\end{cor}
%
%

%
%

\section{The density $\rho_\ell$ and the Prouhet-Tarry-Escott\\ Problem}
\label{sec:escott}
Theorem \ref{thm:main3} and Lemma \ref{lemma_wsmcm2} unveil an interesting relation of having suitable measures for representing ratios of entire functions as generalised Stieltjes functions and finding finite sequences of numbers which have equal sums of powers up to some fixed degree. The problem of finding such finite sequences of integers is called the Prouhet-Tarry-Escott problem, see \cite{borwein}.

The connection is described in the following lemma. 

\begin{lemma}
\label{lemma:PTE-1}
	Let $a=(a_k)_{k=1}^n$, $b=(b_k)_{k=1}^n$ and $\rho_\ell$ be defined as in Lemma \ref{lemma_wsmcm2}. Then,
	$$
	\sum_{k=1}^na_k^j=\sum_{k=1}^nb_k^j
	$$
	for all $j$ with  $0\leq j\leq \ell$ if and only if $\rho_\ell(t)=0$ for all $t\geq \max\{a_n,b_n\}$.
	\begin{proof}
		For $t\geq \max\{a_n,b_n\}$, we have
		$$\rho_\ell(t)=\sum_{k=1}^n (t-a_k)^{\ell}-\sum_{k=1}^n (t-b_k)^\ell.$$
		The result follows by expanding the above and equating to zero.
	\end{proof}
\end{lemma}

The following result justifies why no sequences $a$ and $b$ can be found such that the ratio in Proposition  \ref{thm:main4} is a generalised Stieltjes function of order $p-\ell+2$ for all entire functions of genus $p$ and negative zeros.

\begin{prop}\label{lemma:impossible}
	Let $a = (a_k)_{k = 1}^n$ and $b
	= (b_k)_{k = 1}^n$ be finite sequences of non-negative numbers such that $b
	\prec_w a$, and suppose that
	\[ \sum_{k = 1}^n b_k = \sum_{k = 1}^n a_k \]
	and
	\[ \sum_{k = 1}^n b^2_k = \sum_{k = 1}^n a^2_k . \]
	Then, $a = b$. 
\end{prop}
\begin{proof}
    The function
	\[ \rho_2 (t) = \sum_{k = 1}^n \mathbbm{1}_{[a_k, \infty)} (t) (t - a_k)^2 -
	\sum_{k = 1}^n \mathbbm{1}_{[b_k, \infty)} (t) (t - b_k)^2  \]
	is $C^1$ (since each summand is of class $C^1$).
Its derivative
	$$\rho_2' (t) = 2 \left( \sum_{k = 1}^n \mathbbm{1}_{[a_k, \infty)} (t) (t -
	a_k) - \sum_{k = 1}^n \mathbbm{1}_{[b_k, \infty)} (t) (t - b_k) \right)
	$$
	is non-negative because $b \prec_w a$, see Lemma \ref{lemma_wsmcm}. Hence $\rho_2       $
	is non-decreasing. But $\rho_2 (t) = 0$ for $t \leqslant a_1$, and also, by Lemma \ref{lemma:PTE-1}, for $t
	\geqslant \max (a_n, b_n)$. Hence $\rho_2 \equiv 0$, and thus $a = b$.
\end{proof}


Let us end this paper by giving 
a sufficient condition for non-negativity of the function $\rho_\ell$ in Theorem \ref{thm:main3}.

\begin{prop}
\label{prop:sufficient}
	Let $a = (a_k)_{k = 1}^n$ and $b= (b_k)_{k = 1}^n$ such that $a_1<b_1$ and
	$$
	\sum_{k=1}^na_k^j=\sum_{k=1}^nb_k^j
	$$
	for all $j$ with  $0\leq j\leq n-1$. Then $\rho_{n-1}\geq 0$.
\end{prop}
\begin{proof}
	In the interval $I=(a_1,\max\{a_n,b_n\})$, the function $\rho_1$ can have at most $n-1$ changes of monotonicity. To see this, we put all the numbers $a_k$ and $b_k$ in increasing order and notice that the  continuous and piecewise linear function $\rho_1$ starts increasing after $a_1$. Each appearance of an $a_k$ makes its slope increase by $1$ and each appearance of a $b_k$ decreases the slope by $1$; since a slope of $0$ does not change the monotonicity we are left with at most $n-1$ changes of monotonicity.
	
	It is easy to show that $\rho_k' = k \rho_{k-1}$ for $k\geq 2$ and repeated applications of this relation yield that the function $\rho_{n-1}$ can have at most $1$ change of monotonicity in $I$. Since $\rho_{n-1}$ takes the value $0$ at both ends of $I$, it cannot change sign inside, and as $a_1<b_1$, we have that $\rho_{n-1}\geq 0$.
\end{proof}
In relation to the Prouhet-Tarry-Escott problem, pairs of integer sequences that satisfy the condition of the lemma above are called ideal solutions. We briefly note that the so called ideal symmetric solutions of length $n$ correspond to symmetries on the $\rho_1$ function, and in particular, if $n$ is odd (i.e. we have an odd ideal symmetric solution), then the graph of $\rho_1$ is symmetric with respect to the middle point $((a_1+b_n)/2,0)$, and if $n$ is even, then it is symmetric with respect to the middle line $x=(a_1+a_n)/2$.
%
%
%

\noindent
Dimitris Askitis\\
Department of Psychology\\
University of Copenhagen\\
Øster Farimagsgade 2A\\
DK-1353 Copenhagen K\\
{\em email}:\hspace{2mm}{\tt dimitrios.askitis@psy.ku.dk}
\bigskip

\noindent
Henrik Laurberg Pedersen\\
Department of Mathematical Sciences\\
University of Copenhagen \\
Universitetsparken 5\\
DK-2100, Denmark\\
{\em email}:\hspace{2mm}{\tt henrikp@math.ku.dk}


\begin{thebibliography}{00}
    \bibitem{Barnes} E.W.~Barnes,
  The theory of the multiple gamma function,
  {\it Trans.\ Camb.\ Philos.\ Soc.} {\bf 19} (1904), 374--425. 
	\bibitem{bkp} C.~Berg, S.~Koumandos and H.L.~Pedersen, Nielsen’s beta function and some infinitely divisible distributions, {\em Math.\ Nach.} {\bf 294} (2021), 426--449.
	\bibitem{borwein} P.~Borwein, Computational excursions in analysis and number theory. 
CMS Books in Mathematics/Ouvrages de Math\'ematiques de la SMC, {\bf 10}.
Springer-Verlag, New York (2002). 
\bibitem{bustoz-ismail}  J.~Bustoz and M.E.H.~Ismail, On gamma function inequalities. Math.\ Comp., {\bf 47} (1986), 659--667.
\bibitem{grinshpan-ismail1}A.Z.~Grinshpan and  M.E.H.~Ismail, Completely monotonic functions involving the gamma and $q$-gamma functions, {\em Proc.~Amer.Math.~Soc.} {\bf 134} (2006), 1153--1160.
	\bibitem{ismail1} M.E.H.~Ismail, Inequalities for gamma and $q$-gamma functions of complex arguments, Anal.\ Appl., \textbf{15} (2017), 641--651.
	\bibitem{kp1}D.B.~Karp and E.G.~Prilepkina, Completely Monotonic Gamma Ratio and Infinitely Divisible H-Function of Fox, {\em Comput.\ Methods\ Funct.\  Theory}, {\bf 16} (2016), 135--153.
	\bibitem{levin} B.Ya.~Levin, Lectures on Entire Functions, Translations of Mathematical Monographs, Americal Mathematical Society, {\bf 150} (1996). 
	\bibitem{hlp} H.L.~Pedersen, Completely monotonic functions
	related to logarithmic derivatives of entire functions, {\em Anal.\ Appl.} {\bf 9} (2011), 409--419. 
	\bibitem{RUIJSENAARS} S.N.M.~Ruijsenaars, On Barnes' multiple zeta and gamma
functions, {\em Adv. Math.}, {\bf 156} (2000), 107--132.
\bibitem{schilling}R.L.~Schilling, R.~Song and Z.~Vondracek, Bernstein Functions
Theory and Applications 2ND Ed. De Gruyter Studies in Mathematics, {\bf 37} (2012).
\bibitem{steutel-van-harn} F.W.~Steutel and K.~Van Harn, {Infinite divisibility of probability distributions on the real line}. Marcel Dekker, Inc., New York - Basel (2004).
	
\end{thebibliography}
\end{document}